\documentclass[draft]{birkjour}
\usepackage[noadjust]{cite}
\usepackage{xcolor}
\RequirePackage[all]{xy}

\newtheorem{theorem}{Theorem}[section]

\newtheorem{lemma}[theorem]{Lemma}
\newtheorem{proposition}[theorem]{Proposition}
\theoremstyle{definition}
\newtheorem{defn}[theorem]{Definition}
\theoremstyle{remark}
\newtheorem{remark}[theorem]{Remark}

\numberwithin{equation}{section}

\newcommand{\RNum}[1]{\uppercase\expandafter{\romannumeral #1\relax}}
\newcommand{\BibTeX}{B\kern-0.1emi\kern-0.017emb\kern-0.15em\TeX}
\newcommand{\XYpic}{$\mathrm{X\kern-0.3em\raisebox{-0.18em}{Y}}$-$\mathrm{pic}\,$}

\newcommand{\cl}{C \kern -0.1em \ell}  

\newcommand{\Rm}{\mathbb{R}^m}

\newcommand{\R}{\mathbb{R}}
\newcommand{\Clm}{\mathcal{C}l_{m-1}}

\newcommand{\Bm}{\mathbb{B}^{m}}
\newcommand{\Sm}{\mathbb{S}^{m-1}}
\newcommand{\Hk}{\mathcal{H}_k}
\newcommand{\Mk}{\mathcal{M}_k}
\newcommand{\Mkk}{\mathcal{M}_{k-1}}
\newcommand{\be}{\boldsymbol{e}}
\newcommand{\boy}{\boldsymbol{y}}

\newcommand{\bo}{\boldsymbol}
\newcommand{\bx}{\boldsymbol{x}}
\newcommand{\bu}{\boldsymbol{u}}
\newcommand{\bov}{\boldsymbol{v}}

\newcommand{\bola}{\bo{a}}
\newcommand{\bt}{\boldsymbol{t}}
\newcommand{\bzeta}{\boldsymbol{\zeta}}
\newcommand{\bbeta}{\boldsymbol{\eta}}

\newcommand{\bgamma}{\boldsymbol{\gamma}}

\newcommand{\baa}{\begin{align*}}
\newcommand{\eaa}{\end{align*}}

\newcommand{\ovp}{\overline{\partial}}
\allowdisplaybreaks

\newcommand{\ed}{\end{document}}
\begin{document}

%
%
%
%
%
%
%
%
%

\title[The higher spin complex $\Pi$-operator]
 {On a higher spin generalization of the complex $\Pi$-operator}
\author[Wanqing Cheng]{Wanqing Cheng}
\address{
School of Mathematics \& Physics,\\
Anhui Jianzhu University, Hefei, P.R. China}
\email{cwq@ahjzu.edu.cn}

\author[Chao Ding]{Chao Ding}
\address{Center for Pure Mathematics, \\School of Mathematical Sciences,\\ Anhui University, Hefei, P.R. China}
\email{cding@ahu.edu.cn}

%


%

\date{\today}
\begin{abstract}
Rarita-Schwinger fields are solutions to the relativistic field equation of spin-$3/2$ fermions in four dimensional flat spacetime, which are important in supergravity and superstring theories. Bure\v s et al. generalized it to arbitrary spin $k/2$ in 2002 in the context of Clifford algebras. In this article, we introduce a higher spin $\Pi$-operator related to the Rarita-Schwinger operator. Further, we investigate norm estimates, mapping properties and the adjoint operator of the higher spin $\Pi$-operator. As an application, a higher spin Beltrami equation is introduced, and existence and uniqueness of solutions to this higher spin Beltrami equation is established by the norm estimate of the higher spin $\Pi$-operator. 
\end{abstract}
\label{page:firstblob}
\maketitle
\section{Introduction}~\par
It is well-known that complex analysis is closely related to the theory of partial differential equations, and with functional analytical methods many results of classical function theory can be applied for solving partial differential equations. In particular, the technique of transforming partial differential equations into some integral equations is one of the most commonly used methods. The most important integral transforms in this part of complex analysis are known as the Teodorescu transform, the $\Pi$-operator, and the Beltrami equation. The Beltrami equation is a generalization of Cauchy-Riemann's equation, which has a wide applications in the fields of hydrodynamics \cite{hydro}, electrodynamics \cite{electro} and modern control theory \cite{control}. Formally, the Beltrami equation in the complex plane is given by
\begin{align*}
\frac{\partial f}{\partial \overline{z}}(z)=\mu(z)\frac{\partial f}{\partial {z}}(z),
\end{align*}
where $\mu(z)$ is a measurable function, which is usually called the complex characteristic of $f$. The problem of existence and uniqueness of homogeneous solutions of the Beltrami equation ha been one of the hot topics considered by many mathematicians. 
The first systematic study on this problem in the complex plane was given by Vekua in 1962 \cite{Vekua}. For recent developments in the theory of Beltrami equations, the readers are referred to the monograph written by Iwaniec and Martin \cite{iwaniec}, the book written by Gutlyanskii et al.  \cite{gut} and references therein.
\par
In the last decades, the Teodorescu transform and the $\Pi$-operator has been successfully generalized to higher dimensional Euclidean spaces in the framework of Clifford analysis. One of the first generalizations can be found in \cite{Sp1}. In \cite{Gur1,Gur2} G\"urlebeck and Spr\"o\ss ig developed an operator calculus in real Clifford algebras with special emphasis on Teodorescu transforms and their applications in solving certain boundary value problems. Spr\"o\ss ig \cite{Sp2} studied some elliptic boundary value problems with the Teodorescu transform. One of the important applications of the Teodorescu transform is to give the existence of solutions to the Beltrami equation, and there are many papers contributed to this topic. For instance, K\"ahler \cite{Kahler} generalized the Beltrami equation in cases of quaternions. G\"urlebeck and K\"ahler \cite{GK} investigated a hypercomplex generalization of the complex $\Pi$-operator and the solution of a hypercomplex Beltrami equation. More related work can be found, for instance, in \cite{BRAK,CK,GKS}.
\par
The higher spin theory in the context of Clifford analysis (also known as higher spin Clifford analysis) was firstly introduced by Bure\v s et al. \cite{Bures} in 2002. In \cite{Bures}, the authors introduced the generalized Rarita-Schwinger operator, which acts on functions taking values in $k$-homogeneous monogenic (null solutions to the Dirac operator) polynomials. Further, the Rarita-Schwinger operator can also be considered as a Stein-Weiss gradient as presented in \cite{DWR}. It turns out the Rarita-Schwinger operator is the analog of the Dirac operator in the higher spin cases. Recently, in \cite{DingT}, the Teodorescu transform and its mapping properties in the theory of higher spins have been studied. Therefore, it is natural to define a higher spin $\Pi$-operator and a higher spin Beltrami equation with the Rarita-Schwinger operator and the Teodorescu transform, which are the main topics studied in this article.
\par
This article is organized as follows. In Section 2, we introduce some definitions and notations in Clifford analysis and the higher spin Clifford analysis. Further, some Rarita-Schwinger type operators and function spaces in the higher spin cases are also reviewed here. Section 3 is devoted to a review of some integral formulas and properties of the Rarita-Schwinger operators. In Section 4, we firstly define the higher spin $\Pi$-operator, and an integral representation of the higher spin $\Pi$-operator is given, which leads to a norm estimate of the higher spin $\Pi$-operator. Furthermore, some mapping properties of the higher spin $\Pi$-operator are given here as well. At the end, we introduce a higher spin Beltrami equation, and existence and uniqueness of solutions to the higher spin Beltrami equation have been established by the norm estimate of the higher spin $\Pi$-operator as an application.
\section{Definitions and notations}
In this section, we review some definitions and preliminary results on Clifford analysis and higher spin Clifford analysis, for more details, we refer the readers to \cite{2,10,21}.
\subsection{Clifford algebras}
Let $\R^{m-1}$ be the $(m-1)$-dimensional Euclidean space with a standard orthonormal basis $\{{\be}_1,\ldots,{\be}_{m-1}\}$. The real Clifford algebra $\Clm$ is generated by $\Rm$ with the following relationship
$${\be}_i{\be}_j+{\be}_j{\be}_i=-2\delta_{ij}\be_0,$$
where $\delta_{ij}$ is the Kronecker delta function, and $\be_0=1$ is the identity element. Hence a Clifford number $x\in \Clm$ can be written as $x=\sum_{A}x_Ae_A$ with real coefficients and $A\subset\{1,\ldots,m-1\}$. This suggests that one can consider $\Clm$ as a vector space with dimension $2^{m-1}$. Therefore, a reasonable norm for a Clifford number $x=\sum_{A}x_Ae_A$ should be $|x|=(\sum_{A}x_A^2)^{\frac{1}{2}}$. If we denote $\Clm^{k}=\{x\in \Clm:\ x=\sum_{|A|=k}x_Ae_A\}$, where $|A|$ stands for the cardinality of the set $A$, then one can see that $\Clm=\displaystyle\bigoplus_{k=0}^{m-1} \Clm^k$. For an arbitrary Clifford number ${x} =\sum_{\lvert A\rvert=k}x_{A}{e}_{A}$, we define the Clifford conjugation of ${x}$ by
$$\overline{{x}}=\sum_{A}(-1)^{\frac{\lvert A\rvert(\lvert A \rvert+1)}{2}}x_{A}{e}_{A}. $$
In this article, the $m$-dimensional Euclidean space $\R^{m}$ is identified with $ \Clm^0\oplus\Clm^1$ as following
\begin{align*}
\R^{m}&\longrightarrow  \Clm^0\oplus\Clm^1,\\
(x_0,\ldots,x_{m-1})&\longmapsto \bx=:x_0e_0+x_1e_1+\cdots+x_{m-1}e_{m-1},
\end{align*}
and we usually call $\bx$ a paravector. Given $\bx,\boy\in\R^m$, it is easy to check that $\bx\overline{\boy}=\langle \bx,\boy\rangle-\bx\wedge\overline{\boy}$ and $\bx\wedge\overline{\boy}=-\boy\wedge\overline{\bx}$,
where 
\begin{align*}
\bx\wedge\overline{\boy}=\sum_{0\leq i<j\leq m-1}e_ie_j(x_iy_j-x_jy_i),
\end{align*}
which gives us the following
\begin{align*}
\bx\overline{\boy}+\boy\overline{\bx}=2\langle \bx,\boy\rangle,\ 
\bx\overline{\boy}-\boy\overline{\bx}=-2\bx\wedge\overline{\boy}.
\end{align*}
Given $\bo{a}\in\Sm$ and $\bx\in\R^m$, we have
\begin{align*}
\bo{a}\bx\bo{a}=\bo{a}(-\bar{\bo{a}}\bar{\bx}+2\langle\bo{a},\overline{\bx}\rangle)=-\overline{\bx}+2\langle\bo{a},\overline{\bx}\rangle\bo{a},
\end{align*}
hence, $\bo{a}\bx\bo{a}$ is a reflection of $-\overline{\bx}$ with respect to the hyperplane that is orthogonal to $\bo{a}$.
\par
The generalized Cauchy-Riemann operator in $\R^m$ is given by 
\begin{align*}
\overline{\partial}_{\bx}:=\partial_{x_0}+\sum_{j=1}^{m-1}\be_j\partial_{x_j},
\end{align*}
and $D_{\bx}=\sum_{j=1}^{m-1}\be_j\partial_{x_j}$ is usually called the Dirac operator in $R^{m-1}$, where $\partial_{x_j}$ stands for the partial derivative with respect to $x_j$. The generalized conjugate Cauchy-Riemann operator is denoted by
\begin{align*}
{\partial}_{\bx}:=\partial_{x_0}-\sum_{j=1}^{m-1}\be_j\partial_{x_i}.
\end{align*}

It is easy to  verify that $\overline{\partial}_{\bx}{\partial}_{\bx}={\partial}_{\bx}\overline{\partial}_{\bx}=\Delta_{\bx}$, where $\Delta_x$ is the Laplacian in $\mathbb{R}^m$. Therefore, we usually say Clifford analysis is a refinement of harmonic analysis.
\begin{defn}
 A $\Clm$-valued function $f(x)$ defined on a domain $U$ in $\Rm$ is \emph{left monogenic} if $\overline{\partial}_{\bx}f(\bx)=0$. 
 \end{defn}
 Since Clifford multiplication is not commutative in general, there is a similar definition for right monogenic functions. Sometimes, we will consider the generalized Cauchy-Riemann operator $\overline{\partial}_{\bu}$ in a vector $\bu$ rather than $\bx$. In this article, we use monogenic to represent left monogenic unless specified.
\subsection{Rarita-Schwinger operators}
In theoretical physics, the Rarita–Schwinger equation is the relativistic field equation of spin-$3/2$ fermions in a four-dimensional flat spacetime. It is similar to the Dirac equation for spin-$1/2$ fermions. This equation was first introduced by William Rarita and Julian Schwinger in 1941 \cite{Rarita}. In 2002, Bure\v{s} et al. \cite{Bures} generalized the Rarita-Schwinger operator from spin -$3/2$ to a general spin-$k/2$ in the framework of Clifford analysis. Now, we review the definition and some properties of the Rarita-Schwinger operator as follows. For more details, see \cite{Bures,DJR}.
\par
Let $\mathcal{P}_k(\bu)$ be the space of $\mathcal{C}l_m$-valued polynomials homogeneous of degree $k$ in the variable $\bu$. We denote the following three important spaces of certain homogeneous polynomials.
\begin{align*}
&\mathcal{M}^{+}_k(\bu)=\mathcal{P}_k(\bu)\cap\ker\overline{\partial}_{\bu},\  \mathcal{M}^{-}_k(\bu)=\mathcal{P}_k(\bu)\cap\ker{\partial}_{\bu},\\ 
&\Hk(\bu)=\mathcal{P}_k(\bu)\cap\ker\Delta_{\bu}.
\end{align*}
 Then, we have a Fischer decomposition for $\Hk(\bu)$ as follows. Indeed, there is also a similar Fischer decomposition of $\Hk$ with respect to the Dirac operator, see \cite{DJR}.
\begin{proposition}[Fischer decomposition]\label{fischer}~\\
The space of $\Clm$-valued $k-$homog-eneous harmonic polynomials $\Hk(\bu)$ can be decomposed as
\begin{align*}
&\Hk(\bu)=\Mk^+(\bu)\oplus\overline{\bu}\Mkk^-(\bu),\\
&\Hk(\bu)=\Mk^-(\bu)\oplus{\bu}\Mkk^+(\bu).
\end{align*}
\end{proposition}
\begin{proof}
Here, we only prove the first identity, and the second one can be verified similarly. Let $h_k\in\Hk(\bu)$, then, it is easy to see that $\overline{\partial}_{\bu}h_k(\bu)\in\mathcal{M}_{k-1}^-(\bu)$. Let $p_{k-1}\in\Mkk^-(\bu)$,  a straightforward  calculation shows us that
  \begin{align*}
  \overline{\partial}_{\bu}\overline{\bu} p_{k-1}(\bu)=&(m+2\mathbb{E}_{\bu}-\bu\partial_{\bu})p_{k-1}(\bu)=(m+2\mathbb{E}_{\bu})p_{k-1}(\bu)\\
  =&(m+2k-2)p_{k-1}(\bu).
 \end{align*} 
 Now, we let
 \begin{align*}
 p_{k-1}(\bu)=(m+2k-2)^{-1}\overline{\partial}_{\bu}h_k(\bu),\ p_k(\bu)=h_k(\bu)-\overline{\bu}p_{k-1}(\bu).
 \end{align*}
 One can easily see that $\overline{\partial}_{\bu}p_k(\bu)=0$, and $p_k(\bu)\in\Mk^+(\bu)$. Further, $$\overline{\partial}_{\bu}\overline{\bu} p_{k-1}(\bu)=(m+2k-2)p_{k-1}(\bu)$$ also justifies the fact that $\Mk^+(\bu)\cap\overline{\bu}\Mkk^-(\bu)=\emptyset$, which completes the proof.
\end{proof}
The Fischer decomposition above gives us the following two projections. 
\begin{align*}
&P_k^+: \mathcal{H}_k(\bu)\longrightarrow \mathcal{M}^+_k(\bu),\ P_k^-: \mathcal{H}_k(\bu)\longrightarrow \mathcal{M}^-_k(\bu)\\
&P_k^+=1-\frac{\overline{\bu}\overline{\partial}_{\bu}}{m+2k-2},\ P_k^-=1-\frac{{\bu}{\partial}_{\bu}}{m+2k-2}.
\end{align*}
Suppose $\Omega$ is a domain in $\mathbb{R}^m$, we consider a differentiable function $f: \Omega\times \mathbb{R}^m\longrightarrow \Clm$
such that, for each $\bx\in \Omega$, $f(\bx,\bu)$ is a left monogenic polynomial homogeneous of degree $k$ in $\bu$. Then, the first order conformally invariant differential operator in higher spin theory, named as the \emph{Rarita-Schwinger operator} \cite{Bures,DJR}, is defined by 
\begin{align*}
&R_k:\ C^1(\Omega\times\Bm,\Mk^+(\bu))\longrightarrow C(\Omega\times\Bm,\Mk^+(\bu)),\\
&R_kf(\bx,\bu):=P_k^+\overline{\partial}_{\bx}f(\bx,\bu)=\bigg(1+\frac{\overline{\bu} \overline{\partial}_{\bu}}{m+2k-2}\bigg)\ovp_{\bx}f(\bx,\bu).
\end{align*}
We also define 
\begin{align*}
&R_k^{\dagger}:\ C^1(\Omega\times\Bm,\Mk^-(\bu))\longrightarrow C(\Omega\times\Bm,\Mk^-(\bu)),\\
&R_k^{\dagger}g(\bx,\bu):=P_k^-{\partial}_{\bx}g(\bx,\bu)=\bigg(1+\frac{{\bu} {\partial}_{\bu}}{m+2k-2}\bigg)\partial_{\bx}g(\bx,\bu).
\end{align*}
The identities of $\text{End}(\Mk^+(\bu))$ and $\text{End}(\Mk^-(\bu))$ can be represented by the reproducing kernel $Z_k^+(\bu,\bov)$ and $Z_k^-(\bu,\bov)$, respectively, for the inner spherical monogenics of degree $k$. This so called zonal spherical monogenic satisfies
\begin{eqnarray*}
f(\bu)=\int\displaylimits_{\Sm}{Z_k^{\pm}(\bu,\bov)}f(\bov)dS(\bov),\ \text{for\ all}\ f\in\Mk^{\pm}(\bu).
\end{eqnarray*}
Thanks to the result \cite[Theorem 3.2]{DeBieK}, one can apply a similar argument to obtain that 
\begin{align*}
&Z_k^+(\bu,\bov)=\frac{2\mu+k}{2\mu}|\bu|^k|\bov|^kC_k^{\mu}(t)+\bu\wedge\overline{\bov}|\bu|^{k-1}|\bov|^{k-1}C_{k-1}^{\mu+1}(t),\\
&Z_k^-(\bu,\bov)=\frac{2\mu+k}{2\mu}|\bu|^k|\bov|^kC_k^{\mu}(t)+\overline{\bu}\wedge\bov|\bu|^{k-1}|\bov|^{k-1}C_{k-1}^{\mu+1}(t),
\end{align*}
where $\mu=\frac{m-2}{2}$, $t=\frac{\langle \bu,\bov\rangle}{|\bu||\bov|}$ and $C_k^{\mu}(t)$ is the Gegenbauer polynomial of degree $k$ and order $\mu$ given by
\begin{align}\label{Gegenbauer}
C_k^{\mu}(t)=\sum_{n=0}^{[k/2]}(-1)^n
\frac{\Gamma(k-n+\mu)}{\Gamma(\mu)n!(k-2n)!}(2t)^{k-2n},
\end{align}
which satisfies $\frac{d}{dt}C_k^{\mu}(t)=2\mu C_{k-1}^{\mu+1}(t)$. We also observe that $Z_k^+(\bu,\bov)=Z_k^-(\overline{\bu},\overline{\bov})$, and
 \begin{align}\label{kernelconjugate}
 \overline{Z_k^+(\bu,\bov)}=&\frac{2\mu+k}{2\mu}|\bu|^k|\bov|^kC_k^{\mu}(t)+\overline{\bu\wedge\overline{\bov}}|\bu|^{k-1}|\bov|^{k-1}C_{k-1}^{\mu+1}(t)\nonumber\\
 =&\frac{2\mu+k}{2\mu}|\bu|^k|\bov|^kC_k^{\mu}(t)-\bu\wedge\overline{\bov}|\bu|^{k-1}|\bov|^{k-1}C_{k-1}^{\mu+1}(t)\nonumber\\
 =&\frac{2\mu+k}{2\mu}|\bu|^k|\bov|^kC_k^{\mu}(t)+\bov\wedge\overline{\bu}|\bu|^{k-1}|\bov|^{k-1}C_{k-1}^{\mu+1}(t)\nonumber\\
 =&Z_k^+(\bov,\bu).
 \end{align}
The fundamental solution for $R_k$ and $R_k^{\dagger}$ are given by
\begin{align*}
&E_{k}(\bx,\bu,\bov)=\frac{1}{c_{m,k}}\frac{\overline{\bx}}{|\bx|^m}Z_k^+\bigg(\frac{\bx\bu\bx}{|\bx|^2},\bov\bigg),\\ 
&E^{\dagger}_{k}(\bx,\bu,\bov)=\frac{1}{c_{m,k}}\frac{{\bx}}{|\bx|^m}Z_k^-\bigg(\frac{\overline{\bx}\bu\overline{\bx}}{|\bx|^2},\bov\bigg),
\end{align*}
respectively, where $c_{m,k}=\displaystyle\frac{(m-2)\omega_{m-1}}{m+2k-2}$ and $\omega_{m-1}$ is the area of $(m-1)$-dimensional unit sphere. Given $\bo{a}\in\Sm$, we notice that
\begin{align*}
(\bu\wedge\overline{\bo{a}\bov\bo{a}})\bola=\frac{\bu\bar{\bola}\bar{\bov}\bar{\bola}-\bola\bov\bola\overline{\bu}}{2}\bola=\bola\frac{\bar{\bola}\bu\bar{\bola}\overline{\bov}-\bov\bola\overline{\bu}\bola}{2}=\bola(\bar{\bola}\bu\bar{\bola})\wedge\overline{\bov},
\end{align*}
which gives us a crucial identity as the following
\begin{align*}
Z_k^{\pm}(\bu,\bo{a}\bov\bo{a})\bola=\bola Z_k^{\pm}(\overline{\bola}\bu\overline{\bola},\bov).
\end{align*}
\subsection{Function spaces in higher spin Clifford analysis}
In this subsection, we introduce generalizations of some classical function spaces in higher spin Clifford analysis. Let $\Omega\subset\R^m$ be a domain, the norm of the space $L^p(\Omega\times \Bm,\Clm)$ with $1\leq p<\infty$ is given by
\begin{align*}
	||f||_{L^p}:=\bigg(\int_{\Omega}\int_{\Bm}|f(\bx,\bu)|^pdS(\bu)d\bx\bigg)^{\frac{1}{p}},
\end{align*}
and when $p=\infty$, the function space is equipped with the essential norm. We introduce the Sobolev space $W^{s,t}_p(\Omega\times \Bm,\Clm)$ as the subset of functions $f$ in $L^p(\Omega\times \Bm,\Clm)$ such that $f$ and its derivatives up to order-$(s,t)$ with respect to $(\bx,\bu)$ have a finite $L^p$ norm. Let $\bo{\alpha}=(\alpha_1,\ldots,\alpha_m),\bo{\beta}=(\beta_1,\ldots,\beta_m)$ be multi-indices, $|\bo{\alpha}|=\sum_{j=1}^m\alpha_j,\ |\bo{\beta}|=\sum_{j=1}^m\beta_j$, and $\partial_{\bx}^{\bo{\alpha}}:=(\partial^{\alpha_1}_{x_1},\ldots,\partial_{x_m}^{\alpha_m})$, $\partial_{\bu}^{\bo{\beta}}:=(\partial^{\beta_1}_{u_1},\ldots,\partial_{u_m}^{\beta_m})$. The norm for $W^{s,t}_p(\Omega\times \Bm,\Clm)$ with $p>1$ is given by
\begin{align*}
||f||_{W^{s,t}_p(\Omega\times \Bm,\Clm)}:=\bigg[\sum_{|\bo{\alpha}|=0}^s\sum_{|\bo{\beta}|=0}^t||\partial_{\bx}^{\bo{\alpha}}\partial_{\bu}^{\bo{\beta}}f||^p_{L^p(\Omega\times\Bm,\Clm)}\bigg]^{\frac{1}{p}}.
\end{align*}
Now, we say that a function $f(\bx,\bu)\in L^p(\Omega\times \Bm,\Mk^+(\bu))$ if for each fixed $\bx\in\Omega$, $f(\bx,\bu)\in\Mk^+(\bu)$ with respect to $\bu$ and the $L^p$ norm of $f$ is finite. For convenience, we modify the norm of $L^p(\Omega\times \Bm,\Mk^+(\bu))$ to
\begin{align*}
	||f||_{L^p(\Omega\times\Bm,\Mk^+(\bu))}:=\bigg(\int_{\Omega}\int_{\Sm}|f(\bx,\bu)|^pdS(\bu)d\bx\bigg)^{\frac{1}{p}},
\end{align*}
which is equivalent to the norm $\|\cdot\|_{L^p}$ up to a multiplicative constant due to the homogeneity  of the second variable $\bu$.
\par
In later sections, for Sobolev spaces $W^{s,t}_p(\Omega\times \Bm,\Mk^+(\bu))$, since it is a space of $k$-homogeneous polynomials in the variable $\bu$, we can omit the regularity with respect to $\bu$ to $W^{s}_p(\Omega\times \Bm,\Mk^+(\bu))$ instead. Similar as \cite[Theorem 3.1]{DingT}, we notice that $\Mk^+(\bu)$ is a finite dimensional space, which implies that a Cauchy sequence in $L^p(\Omega\times \Bm,\Mk^{\pm}(\bu))$ is equivalent to finitely many Cauchy sequence in $L^p(\Omega,\Clm)$.  Therefore, we can see that
\begin{proposition}
	The space of $L^p(\Omega\times \Bm,\Mk^{\pm}(\bu))$ is a closed subspace of $L^p(\Omega\times \Bm,\Clm)$.
	\end{proposition}
\begin{proof}
	Let $\{f_n\}_{n=1}^{\infty}$ be a Cauchy sequence in $L^p(\Omega\times \Bm,\Mk^{\pm}(\bu))$, which converges to $f\in L^p(\Omega\times \Bm,\Clm)$. This means that 
	\begin{align*}
		\lim_{n\rightarrow\infty}\int_{\Omega}\int_{\Sm}|f_n(\bx,\bu)-f(\bx,\bu)|^pdS(\bu)d\bx=0,
	\end{align*}
	and this gives rise to that for almost every $(\bx,\bu)\in\Omega\times \Bm$, we have $$\lim_{n\rightarrow\infty}f_n(\bx,\bu)=f(\bx,\bu).$$
	Further, since for each fixed $\bx\in\Omega$, $f_n(\bx,\bu)\in \Mk^{\pm}(\bu)$ and $\Mk^{\pm}(\bu)$ is a finite dimensional vector space, let $\{\varphi_j(\bu)\}_{j=1}^s$ be an orthonormal basis of $\Mk^{\pm}(\bu)$. Then, for each fixed $\bx\in\Omega$, we can rewrite 
	\begin{align*}
	f_n(\bx,\bu)=\sum_{j=1}^s\varphi_{j}(\bu)a_{n,j}(\bx),
	\end{align*}
	where $a_{n,j}(\bx)$ are all real-valued numbers depending on $\bx$. Apparently, the convergence of $\{f_n(\bx,\bu)\}_{n=1}^{\infty}$ is equivalent to the convergence of the sequence $\{a_{n,j}(\bx)\}_{n=1}^{\infty}$ for all $j=1,2,\cdots,s$, we assume that $\displaystyle\lim_{n\rightarrow\infty}a_{n,j}(\bx)=a_j(\bx)$ for $j=1,2,\cdots,s$. Therefore, we must have 
		\begin{align*}
	f(\bx,\bu)=\sum_{j=1}^s\varphi_{j}(\bu)a_{j}(\bx),
	\end{align*}
	which implies that $f\in\Mk^{\pm}(\bu)$, i.e., $f\in L^p(\Omega\times \Bm,\Mk^{\pm}(\bu))$.
\end{proof}
\section{Integral formulas for Rarita-Schwinger operators}
In this section, we introduce some needed integral formulas for $R_k$ and $R_k^{\dagger}$, which can be obtained with similar arguments for the Rarita-Schwinger operators given in \cite{Bures,DJR}.
\begin{defn}
For any $\Clm$-valued polynomials $P(\bu),Q(\bu)$, the inner product
$(P(\bu), Q(\bu))_{\bu}$ with respect to $\bu$ is given by
\begin{align*}
(P(\bu), Q(\bu))_{\bu} =\int_{\Sm}P(\bu)Q(\bu)dS(\bu),
\end{align*}
where $dS(\bu)$ is the area element on the unit sphere $\Sm$.
\end{defn}
Firstly, we introduce a Stokes' Theorem for $\overline{\partial}_{\bx}$ and $\partial_{\bx}$ as follows.
\begin{theorem}[Stokes' Theorem for the Dirac operator]\cite[Theorem A.2.22]{23}
Let $f,g\in C^1(\Omega\times\Bm,\Mk^{\pm}(\bu))\cap C(\overline{\Omega}\times\Bm,\Mk^{\pm}(\bu))$. Then, we have
\begin{align*}
&\int_{\partial\Omega}g(\bx, \bu)d\sigma_{\bx}f(\bx, \bu)\\
=&\int_{\Omega}\big(g(\bx, \bu)\overline{\partial}_{\bx}\big) f(\bx, \bu)d\bx+\int_{\Omega}g(\bx, \bu)\big( \overline{\partial}_{\bx}f(\bx, \bu)\big)d\bx,\\
&\int_{\partial\Omega}g(\bx, \bu)\overline{d\sigma_{\bx}}f(\bx, \bu)\\
=&\int_{\Omega}\left(g(\bx, \bu)\partial_{\bx}\right) f(\bx, \bu)d\bx+\int_{\Omega}g(\bx, \bu)\left(\partial_{\bx}f(\bx, \bu)\right)d\bx,
\end{align*}
where $d\sigma_{\bx}=n(\bx)d\sigma(\bx)$, $n(\bx)$ is the outward unit normal vector on $\partial\Omega$ and $d\sigma(\bx)$ is the area element on $\partial\Omega$.
\end{theorem}
Now, we can establish a Stokes' Theorem for Rarita-Schwinger operators as follows. A similar result for the Rarita-Schwinger operator defined with the Dirac operator is given in \cite{Bures,DJR}.
\begin{theorem}[Rarita-Schwinger Stokes' Theorem]\label{Stokes}
Suppose that the functions $f,g\in C^1(\Omega\times\Bm,\Mk^{\pm}(\bu))\cap C(\overline{\Omega}\times\Bm,\Mk^{\pm}(\bu))$. Then, we have
\begin{align*}
&\int_{\partial\Omega}\left(g(\bx, \bu)d\sigma_{\bx}f(\bx, \bu)\right)_{\bu}\\
=&\int_{\Omega}\left(g(\bx, \bu)R_k, f(\bx, \bu)\right)_{\bu}d\bx+\int_{\Omega}\left(g(\bx, \bu), R_kf(\bx, \bu)\right)_{\bu}d\bx,\\
&\int_{\partial\Omega}\left(g(\bx, \bu)\overline{d\sigma_{\bx}}f(\bx, \bu)\right)_{\bu}\\
=&\int_{\Omega}\left(g(\bx, \bu)R_k^{\dagger}, f(\bx, \bu)\right)_{\bu}d\bx+\int_{\Omega}\left(g(\bx, \bu), R_k^{\dagger}f(\bx, \bu)\right)_{\bu}d\bx.
\end{align*}

\end{theorem}
The Borel-Pompeiu formulas for $R_k$ and $R_k^{\dagger}$ are as follows.
\begin{theorem}[Borel-Pompeiu formula] \label{BPF}
Let $f\in C^1(\Omega\times\Bm,\Mk^+(\bu))\cap C(\overline{\Omega}\times\Bm,\Mk^+(\bu))$ and $g\in C^1(\Omega\times\Bm,\Mk^-(\bu))\cap C(\overline{\Omega}\times\Bm,\Mk^-(\bu))$, then, for all $(\boy,\bu)\in\Omega\times\Bm$, we have
\begin{align*}
f(\boy,\bu)
=&\int_{\partial\Omega}\left(E_k(\bx-\boy,\bu,\bov),d\sigma_{\bx}f(\bx,\bov)\right)_{\bov}\\
&-\int_{\Omega}\left(E_k(\bx-\boy,\bu,\bov),R_{k,\bov}f(\bx,\bov)\right)_{\bov}d\bx,\\
g(\boy,\bu)=&\int_{\partial\Omega}\left(E_k^{\dagger}(\bx-\boy,\bu,\bov),\overline{d\sigma_{\bx}}f(\bx,\bov)\right)_{\bov}\\
&-\int_{\Omega}\left(E_k^{\dagger}(\bx-\boy,\bu,\bov),R_{k,\bov}^{\dagger}f(\bx,\bov)\right)_{\bov}d\bx.
\end{align*}
\end{theorem}
We denote
\begin{align*}
&T_kf(\boy,\bu)=-\int_{\Omega}\left(E_k(\bx-\boy,\bu,\bov),f(\bx,\bov)\right)_{\bov}d\bx,\\
&T_k^{\dagger}g(\boy,\bu)=-\int_{\Omega}\left(E_k^{\dagger}(\bx-\boy,\bu,\bov),g(\bx,\bov)\right)_{\bov}d\bx,\\
&F_kf(\boy,\bu)=\int_{\partial\Omega}\left(E_k(\bx-\boy,\bu,\bov),d\sigma_{\bx}f(\bx,\bov)\right)_{\bov},\\
&F_k^{\dagger}g(\boy,\bu)=\int_{\partial\Omega}\left(E_k^{\dagger}(\bx-\boy,\bu,\bov),\overline{d\sigma_{\bx}}g(\bx,\bov)\right)_{\bov},
\end{align*}
$T_k$ and $T_k^{\dagger}$ are well-known as the \emph{Teodorescu transform}, and $F_k,F_k^{\dagger}$ are known as the \emph{Cauchy–Bitsadze operator}, see \cite[Chapter 8]{23}. Therefore, the Borel-Pompeiu formula can be rewritten in the following short form
\begin{align*}
&f(\boy,\bu)=F_kf(\boy,\bu)+T_k(R_kf)(\boy,\bu),\\
&g(\boy,\bu)=F_k^{\dagger}g(\boy,\bu)+T_k^{\dagger}(R_k^{\dagger}g)(\boy,\bu),
\end{align*}
for all $f\in C^1(\Omega\times\Bm,\Mk^+(\bu))\cap C(\overline{\Omega}\times\Bm,\Mk^+(\bu))$ and $g\in C^1(\Omega\times\Bm,\Mk^-(\bu))\cap C(\overline{\Omega}\times\Bm,\Mk^-(\bu))$. Similar as in \cite{DingT,DJR}, $T_k$ and $T_k^{\dagger}$ are the right inverses of $R_k$ and $R_k^{\dagger}$, respectively.
\begin{proposition}\label{Inverse}
Let $f\in C^1(\Omega\times\Bm,\Mk^+(\bu))\cap C(\overline{\Omega}\times\Bm,\Mk^+(\bu))$ and $g\in C^1(\Omega\times\Bm,\Mk^-(\bu))\cap C(\overline{\Omega}\times\Bm,\Mk^-(\bu))$. Then, we have
\begin{align*}
R_kT_kf(\boy,\bu)=f(\boy,\bu),\ R_k^{\dagger}T_k^{\dagger}g(\boy,\bu)=g(\boy,\bu),
\end{align*}
for all $(\boy,\bu)\in\Omega\times\Bm$.
\end{proposition}
We present a mapping property of $R_k$ and $R_k^{\dagger}$ below, which is needed for the later discussion on the norm estimate of $\Pi$-operator.
\begin{proposition}\label{Rkbound}
Let $\Omega\subset\R^{m+1}$ be a bounded domain. Then, we have
\begin{align*}
R_k:\ W_2^1(\Omega\times\Bm,\Mk^-(\bu))\longrightarrow L^2(\Omega\times\Bm,\Mk^+(\bu))
\end{align*}
is bounded.
\end{proposition}
\begin{proof}
In the Fischer decomposition in Proposition \ref{fischer}, we denote $Q_k^+=I-P_k^+$ , which is the projection from $\Hk(\bu)$ to $\overline{\bu}\Mkk^-(\bu)$. Suppose that $g\in\Hk(\bu)$, then we have $P_k^+g\in \Mk^+(\bu)$ and $Q_k^+g\in\overline{\bu}\Mkk^-(\bu)$, which gives that $\overline{P_k^+g}\in\mathcal{M}_{k,r}^-(\bu)$, in other words, $(P_k^+g)\partial_{\bu}=0$. Therefore, the Stokes' Theorem in 
\ref{Stokes} tells us that
\begin{align}\label{Ortho}
\int_{\Sm}\overline{P_k^+g(\bu)}(Q_k^+g(\bu))dS(\bu)=0.
\end{align}
Now, we can see that
\begin{align*}
&\|R_kf\|_{L^2(\Omega\times\Bm,\Mk^-(\bu))}^2=\int_{\Omega}\int_{\Sm}|R_kf(\boy,\bu)|^2dS(\bu)d\boy\\
=&\bigg[\int_{\Omega}\int_{\Sm}\overline{R_kf(\boy,\bu)}R_kf(\boy,\bu)dS(\bu)d\boy\bigg]_0\\
=&\bigg[\int_{\Omega}\int_{\Sm}\overline{P_k^+\overline{\partial}_{\boy}f(\boy,\bu)}P_k^+\overline{\partial}_{\boy}f(\boy,\bu)dS(\bu)d\boy\bigg]_0\\
\leq&\bigg[\int_{\Omega}\int_{\Sm}\overline{(P_k^++Q_k^+)\overline{\partial}_{\boy}f(\boy,\bu)}(P_k^++Q_k^+)\overline{\partial}_{\boy}f(\boy,\bu)dS(\bu)d\boy\bigg]_0\\
=&\bigg[\int_{\Omega}\int_{\Sm}\overline{\overline{\partial}_{\boy}f(\boy,\bu)}\overline{\partial}_{\boy}f(\boy,\bu)dS(\bu)d\boy\bigg]_0\\
=&\int_{\Omega}\int_{\Sm}|\overline{\partial}_{\boy}f(\boy,\bu)|^2dS(\bu)d\boy=\|\overline{\partial}_{\boy}f\|_{L^2(\Omega\times\Bm,\Mk^-(\bu))}^2\\
\leq&m\sum_{i=0}^{m-1}\|\partial_{y_i}f\|_{L^2(\Omega\times\Bm,\Mk^-(\bu))}^2\leq m\|f\|_{W_2^1(\Omega\times\Bm,\Mk^+(\bu))}^2,
\end{align*}
which completes the proof.
\end{proof}
\begin{remark}
The proof of the Proposition above heavily relies on the identity \eqref{Ortho}, which only gives us the results on the $L^2$ instead of $L^p$. This is also the reason that the mapping properties obtained later are all on the $L^2$ space.
\end{remark}
An important technical lemma on homogeneous harmonic polynomials is as follows. 
\begin{lemma}\label{harortho}\cite[Lemma 6]{DJR}
	Suppose $h_k:\Rm\longrightarrow \Clm$ is a harmonic polynomial homogeneous of degree $k$ with $m\geq 2$. Suppose $\bu\in\Sm$, then
	\begin{align*}
		\int_{\Sm}h_k(\bx\bu\bx)dS(\bx)=c_{m,k}h_k(\bu).
	\end{align*}
\end{lemma}
\section{Definitions and properties for the higher spin $\Pi$-operator}
An analog of the complex $\Pi$-operator in higher spin Clifford analysis is the following.
\begin{defn} 
Let $f\in C^1(\Omega\times\Bm,\Mk^+(\bu))\cap C(\overline{\Omega}\times\Bm,\Mk^+(\bu))$, the \emph{higher spin $\Pi$-operator} is given by
\begin{align*}
\Pi f(\boy,\bu):=R_k^{\dagger}T_kf(\boy,\bu).
\end{align*}
\end{defn}
There is also an analog of the higher spin $\Pi$-operator, which is given by
\begin{align*}
\Pi^{\dagger} g(\boy,\bu):=R_kT_k^{\dagger}g(\boy,\bu),
\end{align*}
where $g\in C^1(\Omega\times\Bm,\Mk^-(\bu))\cap C(\overline{\Omega}\times\Bm,\Mk^-(\bu))$. Next, we will give integral expressions for $\Pi$ and $\Pi^{\dagger}$. 
\begin{theorem}[Integral representations]\label{PiIR}
Let $f\in C^1(\Omega\times\Bm,\Mk^+(\bu))\cap C(\overline{\Omega}\times\Bm,\Mk^+(\bu))$ and $g\in C^1(\Omega\times\Bm,\Mk^-(\bu))\cap C(\overline{\Omega}\times\Bm,\Mk^-(\bu))$. Then, we have
\begin{align*}
&\Pi f(\boy,\bu)=R_k^{\dagger}T_kf(\boy,\bu)=P_k^-\partial_{\boy}T_kf(\boy,\bu)\\
=&\int_{\Omega}\int_{\Sm}R_k^{\dagger}E_k(\bx-\boy,\bu,\bov)f(\bx,\bov)dS(\bov)d\bx
+P_k^-f(\boy,\bu)\\
-&c_{m,k}^{-1}P_k^-\int_{\Sm}\frac{2\overline{\bt}}{2-m}\bigg[\overline{\bu}\langle\bt,\overline{\partial}_{\bbeta}\rangle+\langle\bt,\overline{\bu}\rangle\partial_{\bbeta}+2\overline{\bt}\langle\bt,\overline{\bu}\rangle\langle\bt,\overline{\partial}_{\bbeta}\rangle\bigg]f(\boy,\bt\bu\bt)d\sigma(\bt),\\
&\Pi^{\dagger} f(\boy,\bu)=R_kT_k^{\dagger}f(\boy,\bu)=P_k^+\overline{\partial}_{\boy}T_k^{\dagger}f(\boy,\bu)\\
=&\int_{\Omega}\int_{\Sm}R_kE_k^{\dagger}(\bx-\boy,\bu,\bov)f(\bx,\bov)dS(\bov)d\bx
+P_k^+f(\boy,\bu)\\
-&c_{m,k}^{-1}P_k^+\int_{\Sm}\frac{2{\bt}}{2-m}\bigg[{\bu}\langle\bt,{\partial}_{\bbeta}\rangle+\langle\bt,{\bu}\rangle\overline{\partial}_{\bbeta}+2{\bt}\langle\bt,{\bu}\rangle\langle\bt,{\partial}_{\bbeta}\rangle\bigg]f(\boy,\bt\bu\bt)d\sigma(\bt),
\end{align*}
where $\bbeta=\displaystyle\frac{(\bx-\boy)\bu(\bx-\boy)}{|\bx-\boy|^2},\bo{t}=\frac{\bx-\boy}{|\bx-\boy|}$.
\end{theorem}
\begin{proof}
To obtain an integral expression for $\Pi$, the main work is to find the derivative of $T_kf$. The strategy is similar as applied in \cite[Theorem 5.3]{DingT}. The calculation there is rather long, we only show the main steps here. Recall that
\begin{align*}
&T_kf(\boy,\bu)=-\int_{\Omega}\int_{\Sm}E_k(\bx-\boy,\bu,\bov),f(\bx,\bov) dS(\bov)d\bx,\\
&E_{k}(\bx-\boy,\bu,\bov)=\frac{1}{c_{m,k}}\frac{\overline{\bx-\boy}}{|\bx-\boy|^m}Z_k^+\bigg(\frac{(\bx-\boy)\bu(\bx-\boy)}{|\bx-\boy|^2},\bov\bigg),\\
&\frac{\overline{\bx-\boy}}{|\bx-\boy|^m}=\partial_{\bx}\frac{|\bx-\boy|^{2-m}}{2-m},
\end{align*}
and let $\epsilon>0$ be sufficiently small such that $B(\boy,\epsilon)\subset\Omega$ and $\Omega_{\epsilon}=\Omega\backslash B(\boy,\epsilon)$. Furthermore, let $\bbeta=\frac{(\bx-\boy)\bu(\bx-\boy)}{|\bx-\boy|^2}$, with Stokes' Theorem of $\partial_{\bx}$, the fact that $Z_k^+$ is the reproducing kernel for $\Mk^+(\bu)$ and the argument in \cite[Theorem 5.3]{DingT}, we can have
\begin{align*}
c_{m,k}T_kf(\boy,\bu)=&\int_{\Omega}\frac{|\bx-\boy|^{2-m}}{2-m}\big(\partial_{\bx}f(\bx,\bbeta)\big)d\bx\\
&-\int_{\partial\Omega}\frac{|\bx-\boy|^{2-m}}{2-m}\overline{n(\bx)}f(\bx,\bbeta)d\sigma(\bx).
\end{align*}
and
\begin{align}
&c_{m,k}\frac{\partial}{\partial y_i}T_kf(\boy,\bu)\nonumber\\
=&\int_{\Omega}\frac{y_i-x_i}{|\bx-\boy|^m}(\partial_{\bx}f(\bx,\bbeta))+\frac{|\bx-\boy|^{2-m}}{2-m}\sum_{s=0}^{m-1}\frac{\partial\eta_s}{\partial y_i}\frac{\partial f(\bx,\bbeta)}{\partial\eta_s}d\bx\nonumber\\
&-\int_{\partial\Omega}\frac{y_i-x_i}{|\bx-\boy|^m}\overline{n(\bx)}f(\bx,\bbeta)+\frac{|\bx-\boy|^{2-m}}{2-m}\overline{n(\bx)}\sum_{s=0}^{m-1}\frac{\partial\eta_s}{\partial y_i}\frac{\partial f(\bx,\bbeta)}{\partial\eta_s}d\sigma(\bx)\nonumber\\
=&\int\limits_{\substack{\partial B(\boy,\epsilon)\\\epsilon\rightarrow 0}}\frac{x_i-y_i}{|\bx-\boy|^m}\overline{n(\bx)}f(\bx,\bbeta)-\frac{|\bx-\boy|^{2-m}}{2-m}\overline{n(\bx)}\sum_{s=0}^{m-1}\frac{\partial\eta_s}{\partial y_i}\frac{\partial f(\bx,\bbeta)}{\partial\eta_s}d\sigma(\bx)\nonumber\\
&-\int\limits_{\substack{\Omega_{\epsilon}\\\epsilon\rightarrow 0}}\bigg(\frac{y_i-x_i}{|\bx-\boy|^m}\partial_{\bx}\bigg)f(\bx,\bbeta)+\bigg(\frac{|\bx-\boy|^{2-m}}{2-m}\partial_{\bx}\bigg)\sum_{s=0}^{m-1}\frac{\partial\eta_s}{\partial y_i}\frac{\partial f(\bx,\bbeta)}{\partial\eta_s}d\bx\nonumber\\
&+\int\limits_{\substack{B(\boy,\epsilon)\\ \epsilon\rightarrow 0}}\frac{y_i-x_i}{|\bx-\boy|^m}\big(\partial_{\bx}f(\bx,\bbeta)\big)+\frac{|\bx-\boy|^{2-m}}{2-m}+\sum_{s=0}^{m-1}\partial_{\bx}\frac{\partial\eta_s}{\partial y_i}\frac{\partial f(\bx,\bbeta)}{\partial\eta_s}d\bx\nonumber\\
=&:I_1+I_2+I_3.\label{I123}
\end{align}
Since 
\begin{align*}
\bbeta=\frac{(\bx-\boy)\bu(\bx-\boy)}{|\bx-\boy|^2}=-\overline{\bu}+2\frac{\langle \bx-\boy,\overline{\bu}\rangle(\bx-\boy)}{|\bx-\boy|^2},
\end{align*}
a straightforward calculation gives us that
\begin{align}\label{peta}
\frac{\partial\eta_s}{\partial y_i}=\frac{2\lambda(i)u_i(x_s-y_s)-2\langle\bx-\boy,\overline{\bu}\rangle\delta_{si}}{|\bx-\boy|^2}-4\frac{\langle\bx-\boy,\overline{\bu}\rangle(x_s-y_s)(y_i-x_i)}{|\bx-\boy|^4}.
\end{align}
Firstly, if we let $\bx-\boy=\epsilon\bt$ with $r>0,\bt\in\Sm$, then, by the homogeneity of $\bx-\boy$ in $I_3$, we know that $I_3=0$ when $\epsilon\rightarrow 0$. Secondly, we notice that
\begin{align*}
&\frac{y_i-x_i}{|\bx-\boy|^m}\partial_{\bx}=\frac{\partial}{\partial y_i}\partial_{\bx}\frac{|\bx-\boy|^{2-m}}{2-m}=\frac{\partial}{\partial y_i}\frac{\overline{\bx-\boy}}{|\bx-\boy|^m},\\
&f(\bx,\bbeta)=\int_{\Sm}Z_k(\bbeta,\bov)f(\bx,\bov)dS(\bov),
\end{align*}
where \begin{align*}
\lambda(i)=\begin{cases}-1,&\ \text{when}\ i=0,\\ 1,&\ \text{otherwise}.\end{cases}
\end{align*}
Hence, we have
\begin{align*}
I_2=&-\int\limits_{\substack{\Omega_{\epsilon}\\\epsilon\rightarrow 0}}\bigg(\frac{\partial}{\partial y_i}\frac{\overline{\bx-\boy}}{|\bx-\boy|^m}\bigg)f(\bx,\bbeta)+\frac{\overline{\bx-\boy}}{|\bx-\boy|^m}\frac{\partial}{\partial y_i} f(\bx,\bbeta)d\bx\\
=&-\int\limits_{\substack{\Omega_{\epsilon}\\\epsilon\rightarrow 0}}\int_{\Sm}\frac{\partial}{\partial y_i}\bigg[\frac{\overline{\bx-\boy}}{|\bx-\boy|^m}Z_k^+\bigg(\frac{(\bx-\boy)\bu(\bx-\boy)}{|\bx-\boy|^2},\bov\bigg)\bigg]f(\bx,\bov)dS(\bov)d\bx\\
=&-c_{m,k}\int_{\Omega}\int_{\Sm}\frac{\partial}{\partial y_i}E_k(\bx-\boy,\bu,\bov)f(\bx,\bov)dS(\bov)d\bx.
\end{align*}
We substitute $\bx-\boy=\epsilon\bt$ into $I_1$ to obtain
\begin{align*}
I_1=&\lim_{\epsilon\rightarrow 0}\int_{\Sm}t_i\overline{\bt}f(\boy+\epsilon\bt,\bt\bu\bt)\\
&-\frac{2\overline{\bt}}{2-m}\bigg[\lambda(i)u_i\langle\bt,\overline{\partial}_{\bbeta}\rangle+\langle\bt,\overline{\bu}\rangle\frac{\partial}{\partial \eta_i}+2t_i\langle\bt,\overline{\bu}\rangle\langle\bt,\overline{\partial}_{\bbeta}\rangle\bigg]f(\boy,\bt\bu\bt)d\sigma(\bt)\\
=&\int_{\Sm}t_i\overline{\bt}f(\boy,\bt\bu\bt)\\
&-\frac{2\overline{\bt}}{2-m}\bigg[\lambda(i)u_i\langle\bt,\overline{\partial}_{\bbeta}\rangle+\langle\bt,\overline{\bu}\rangle\frac{\partial}{\partial \eta_i}+2t_i\langle\bt,\overline{\bu}\rangle\langle\bt,\overline{\partial}_{\bbeta}\rangle\bigg]f(\boy,\bt\bu\bt)d\sigma(\bt).
\end{align*}
Plugging back to \eqref{I123}, we obtain
\begin{align*}
&\partial_{\boy}T_kf(\boy,\bu)\\
=&-\int_{\Omega}\int_{\Sm}\partial_{\boy}E_k(\bx-\boy,\bu,\bov)f(\bx,\bov)dS(\bov)d\bx
+c_{m,k}^{-1}\int_{\Sm}f(\boy,\bt\bu\bt)d\sigma(\bt)\\
&-c_{m,k}^{-1}\int_{\Sm}\frac{2\overline{\bt}}{2-m}\bigg[\overline{\bu}\langle\bt,\overline{\partial}_{\bbeta}\rangle+\langle\bt,\overline{\bu}\rangle\partial_{\bbeta}+2\overline{\bt}\langle\bt,\overline{\bu}\rangle\langle\bt,\overline{\partial}_{\bbeta}\rangle\bigg]f(\boy,\bt\bu\bt)d\sigma(\bt).
\end{align*}
According to Lemma \ref{harortho}, we know that
\begin{align*}
	II=c_{m,k}^{-1}\int_{\Sm}f(\boy,\bt\bu \bt)d\sigma(\bt)
	=c_{m,k}^{-1}\frac{m-2}{m+2k-2}\omega_{m-1}f(\boy,\bu)=f(\boy,\bu),
\end{align*}
which gives us that 
\begin{align}\label{PartialT}
	&\partial_{\boy}T_kf(\boy,\bu)\nonumber\\
	=&-\int_{\Omega}\int_{\Sm}\partial_{\boy}E_k(\bx-\boy,\bu,\bov)f(\bx,\bov)dS(\bov)d\bx
	+f(\boy,\bu)\nonumber\\
	&-c_{m,k}^{-1}\int_{\Sm}\frac{2\overline{\bt}}{2-m}\bigg[\overline{\bu}\langle\bt,\overline{\partial}_{\bbeta}\rangle+\langle\bt,\overline{\bu}\rangle\partial_{\bbeta}+2\overline{\bt}\langle\bt,\overline{\bu}\rangle\langle\bt,\overline{\partial}_{\bbeta}\rangle\bigg]f(\boy,\bt\bu\bt)d\sigma(\bt).
\end{align}
Therefore, the integral representation of $\Pi$ is given by
\begin{align*}
&\Pi f(\boy,\bu)=R_k^{\dagger}T_kf(\boy,\bu)=P_k^-\partial_{\boy}T_kf(\boy,\bu)\\
=&\int_{\Omega}\int_{\Sm}R_k^{\dagger}E_k(\bx-\boy,\bu,\bov)f(\bx,\bov)dS(\bov)d\bx
+P_k^-f(\boy,\bu)\\
-&c_{m,k}^{-1}P_k^-\int_{\Sm}\frac{2\overline{\bt}}{2-m}\bigg[\overline{\bu}\langle\bt,\overline{\partial}_{\bbeta}\rangle+\langle\bt,\overline{\bu}\rangle\partial_{\bbeta}+2\overline{\bt}\langle\bt,\overline{\bu}\rangle\langle\bt,\overline{\partial}_{\bbeta}\rangle\bigg]f(\boy,\bt\bu\bt)d\sigma(\bt).
\end{align*}
A similar calculation can give us the integral representation of $\Pi^{\dagger}$.
\end{proof}
\begin{remark}
Although the integral representation of the $\Pi$-operator looks rather complicated, it can be used to establish a norm estimate of the higher spin $\Pi$-operator.
\end{remark}
To introduce the norm estimate of the higher spin $\Pi$-operator, we need the following estimates on derivatives of harmonic functions.
\begin{lemma}\cite[Theorem 7, Chap. 2]{Evans} Let $\Omega\subset\Rm$ be a bounded domain and $f$ is harmonic in $\Omega$. Then
\begin{align*}
|D^{\alpha}f(\bx)|\leq \frac{C_s}{r^{m+s}}\|f\|_{L^1(B(\bx,r))}
\end{align*}
for each ball $B(\bx,r)\subset \Omega$ and each multi-index $\alpha=(\alpha_0,\cdots,\alpha_{m-1})$ of order $|\alpha|=s$. Here $D^{\alpha}=\partial_{x_1}^{\alpha_1}\cdots\partial_{x_0}^{\alpha_{m-1}}$,
\begin{align*}
C_0=\frac{1}{V(m)}, C_s=\frac{(2^{m+1}ms)^s}{V(m)},
\end{align*}
and $V(m)$ is the volume of the unit ball in $\Rm$.
\end{lemma}
\begin{remark}
In particular, with H\"older's inequality, the lemma above immediately gives us that
\begin{align}\label{harmonicestimate}
\bigg\vert\frac{\partial}{\partial x_i}f(\bx)\bigg\vert\leq \frac{C_1}{r^{m+1}}V(m)^{\frac{1}{2}}r^{\frac{m}{2}}\|f\|_{L^2(B(\bx,r))}=\frac{m2^{m+1}}{r^{\frac{m}{2}+1}\sqrt{V(m)}}\|f\|_{L^2(B(\bx,r))}.
\end{align}
\end{remark}
Now, we present a norm estimate for the higher spin $\Pi$-operator as follows.
\begin{theorem}[Norm estimates]\label{BdPi}
The higher spin $\Pi$-operator and $\Pi^{\dagger}$-operator
\begin{align*}
&\Pi:\ L^2(\Omega\times\Bm,\Mk^+(\bu))\longrightarrow L^2(\Omega\times\Bm,\Mk^-(\bu)),\\
&\Pi^{\dagger}:\ L^2(\Omega\times\Bm,\Mk^-(\bu))\longrightarrow L^2(\Omega\times\Bm,\Mk^+(\bu))
\end{align*}
are bounded. Furthermore, we have
\begin{align*}
\|\Pi f\|_{L^2(\Omega\times\Bm,\Mk^+(\bu))}\leq C\|f\|_{L^2(\Omega\times\Bm,\Mk^+(\bu))},
\end{align*}
where
\begin{align*}
C=&\sqrt{2(8mC_2+C_1)^2\omega_{m-1}^2+\frac{m^32^{4m+2k+5}}{m+2k}+2},
\\
C_1=&(2m-2)\bigg[\frac{2\mu+k}{2\mu}\sum_{n=0}^{[\frac{k}{2}]}
\frac{\Gamma(k-n+\mu)2^{k-2n}}{\Gamma(\mu)n!(k-2n)!}\\
&\quad\quad\quad\quad+\sum_{n=0}^{[\frac{k-1}{2}]}
\frac{\Gamma(k-n+\mu)2^{k-2n-1}}{\Gamma(\mu+1)n!(k-2n-1)!}\bigg]\\
C_2
=&\sum_{n=0}^{[\frac{k-1}{2}]}
\frac{(2\mu+k)\Gamma(k-n+\mu)2^{k-2n-1}}{\Gamma(\mu+1)n!(k-2n-1)!}\\
&+\sum_{n=0}^{[\frac{k-2}{2}]}
\frac{(2\mu+2)\Gamma(k-n+\mu)2^{k-2n-2}}{\Gamma(\mu+2)n!(k-2n-2)!}.
\end{align*}
\end{theorem}
\begin{proof}
Similar as the proof of Proposition \ref{Rkbound}, in the Fischer decomposition in Proposition \ref{fischer}, we denote $Q_k^-=I-P_k^-$ , which is the projection from $\Hk(\bu)$ to $\bu\Mk^+(\bu)$. Suppose that $g\in\Hk(\bu)$, then we have $P_k^-g\in \Mk^-(\bu)$ and $Q_k^-g\in\bu\Mkk^+(\bu)$, which gives that $\overline{P_k^-g}\in\mathcal{M}_{k,r}^+(\bu)$, in other words, $(P_k^-g)R_k=0$. Therefore, the Stokes' Theorem in 
\ref{Stokes} tells us that
\begin{align*}
\int_{\Sm}\overline{P_k^-g(\bu)}(Q_k^-g(\bu))dS(\bu)=0.
\end{align*}
Now, we can see that
\begin{align}\label{Pinorm}
&\|\Pi f\|_{L^2(\Omega\times\Bm,\Mk^-(\bu))}^2=\|R_k^{\dagger}T_kf\|_{L^2(\Omega\times\Bm,\Mk^-(\bu))}^2\nonumber\\
=&\int_{\Omega}\int_{\Sm}|R_k^{\dagger}T_kf(\boy,\bu)|^2dS(\bu)d\boy\nonumber\\
=&\bigg[\int_{\Omega}\int_{\Sm}\overline{R_k^{\dagger}T_kf(\boy,\bu)}R_k^{\dagger}T_kf(\boy,\bu)dS(\bu)d\boy\bigg]_0\nonumber\\
=&\bigg[\int_{\Omega}\int_{\Sm}\overline{P_k^-{\partial}_{\boy}T_kf(\boy,\bu)}P_k^-{\partial}_{\boy}T_kf(\boy,\bu)dS(\bu)d\boy\bigg]_0\nonumber\\
\leq&\bigg[\int_{\Omega}\int_{\Sm}\overline{(P_k^-+Q_k^-){\partial}_{\boy}T_kf(\boy,\bu)}(P_k^-+Q_k^-){\partial}_{\boy}T_kf(\boy,\bu)dS(\bu)d\boy\bigg]_0\nonumber\\
=&\bigg[\int_{\Omega}\int_{\Sm}\overline{{\partial}_{\boy}T_kf(\boy,\bu)}{\partial}_{\boy}T_kf(\boy,\bu)dS(\bu)d\boy\bigg]_0\nonumber\\
=&\int_{\Omega}\int_{\Sm}|{\partial}_{\boy}T_kf(\boy,\bu)|^2dS(\bu)d\boy.
\end{align}
Secondly, we have the integral expression \eqref{PartialT} as
\begin{align}\label{0norm}
&|\partial_{\boy}T_kf(\boy,\bu)|^2\nonumber\\
=&\bigg\vert\int_{\Omega}\int_{\Sm}\partial_{\boy}E_k(\bx-\boy,\bu,\bov)f(\bx,\bov)dS(\bov)d\bx
-f(\boy,\bu)\nonumber\\
&+c_{m,k}^{-1}\int_{\Sm}\frac{2\overline{\bt}}{2-m}\bigg[\overline{\bu}\langle\bt,\overline{\partial}_{\bbeta}\rangle+\langle\bt,\overline{\bu}\rangle\partial_{\bbeta}+2\overline{\bt}\langle\bt,\overline{\bu}\rangle\langle\bt,\overline{\partial}_{\bbeta}\rangle\bigg]f(\boy,\bt\bu\bt)d\sigma(\bt)\bigg\vert^2\nonumber\\
=:&|I+II+III|^2\leq 2(|I|^2+|II|^2+|III|^2).
\end{align}
To estimate $|I|^2$, a straightforward calculation shows that 
\begin{align*}
\partial_{\boy}E_k(\bx-\boy,\bu,\bov)=&\frac{m-2}{|\bx-\boy|^m}Z_k^+(\bbeta,\bov)+m\frac{(\overline{\bx-\boy})^2}{|\bx-\boy|^{m+2}}Z_k^+(\bbeta,\bov)\\
&+\frac{\overline{\bx-\boy}}{|\bx-\boy|^m}\sum_{s=0}^{m-1}\partial_{\boy}\eta_s\frac{\partial Z_k^+(\bbeta,\bov)}{\partial \eta_s}.
\end{align*}
With the expression of \eqref{peta}, we can find that
\begin{align*}
&|\partial_{\boy}E_k(\bx-\boy,\bu,\bov)|\\
\leq& (m-2)|\bx-\boy|^{-m}|Z_k^+(\bbeta,\bov)|+m|\bx-\boy|^{-m}|Z_k^+(\bbeta,\bov)|\\
&+|\bx-\boy|^{-m}\bigg[2|\partial_{\bbeta}Z_k^+(\bbeta,\bov)|+6\bigg(\sum_{i=0}^{m-1}|\partial_{\eta_i}Z_k^+(\bbeta,\bov)|^2\bigg)^{\frac{1}{2}}\bigg]\\
\leq& (2m-2)|\bx-\boy|^{-m}|Z_k^+(\bbeta,\bov)|+8|\bx-\boy|^{-m}\sum_{i=0}^{m-1}\bigg\vert\frac{\partial Z_k^+(\bbeta,\bov)}{\partial{\eta_i}}\bigg\vert.
\end{align*}
We notice that when $\bbeta,\bov\in\Sm$, we have
\begin{align*}
Z_k^+(\bbeta,\bov)=\frac{2\mu+k}{2\mu}C_k^{\mu}(t)+\bbeta\wedge\overline{\bov}C_{k-1}^{\mu+1}(t),
\end{align*}
where $t=\langle \bbeta,\bov\rangle$. With the expression of $C_k^{\mu}(t)$ given in \eqref{Gegenbauer}, we have
\begin{align*}
&(2m-2)|Z_k^+(\bbeta,\bov)|\leq (2m-2)\bigg[\frac{2\mu+k}{2\mu}|C_k^{\mu}(t)|+|C_{k-1}^{\mu+1}(t)|\bigg]\\
\leq&(2m-2)\bigg[\frac{2\mu+k}{2\mu}\sum_{n=0}^{[\frac{k}{2}]}
\frac{\Gamma(k-n+\mu)2^{k-2n}}{\Gamma(\mu)n!(k-2n)!}+\sum_{n=0}^{[\frac{k-1}{2}]}
\frac{\Gamma(k-n+\mu)2^{k-2n-1}}{\Gamma(\mu+1)n!(k-2n-1)!}\bigg]\\
=:&C_1,\\
&\bigg\vert\frac{\partial Z_k^+(\bbeta,\bov)}{\partial{\eta_i}}\bigg\vert=\bigg\vert\frac{\partial t}{\partial \eta_i}\bigg(\frac{2\mu+k}{2\mu}\frac{dC_k^{\mu}(t)}{dt}+\bbeta\wedge\overline{\bov}\frac{d C_{k-1}^{\mu+1}(t)}{dt}\bigg)\bigg\vert\\
=&\bigg\vert v_i\bigg((2\mu+k)C_{k-1}^{\mu+1}(t)+\bbeta\wedge\overline{\bov}(2\mu+2)C_{k-2}^{\mu+2}(t)\bigg)\bigg\vert\\
\leq &(2\mu+k)|C_{k-1}^{\mu+1}(t)|+(2\mu+2)|C_{k-2}^{\mu+2}(t)|\\
\leq &\sum_{n=0}^{[\frac{k-1}{2}]}
\frac{(2\mu+k)\Gamma(k-n+\mu)2^{k-2n-1}}{\Gamma(\mu+1)n!(k-2n-1)!}+\sum_{n=0}^{[\frac{k-2}{2}]}
\frac{(2\mu+2)\Gamma(k-n+\mu)2^{k-2n-2}}{\Gamma(\mu+2)n!(k-2n-2)!}\\
=:&C_2.
\end{align*}
Therefore, we have that
\begin{align*}
|\partial_{\boy}E_k(\bx-\boy,\bu,\bov)|\leq& C_1|\bx-\boy|^{-m}+8mC_2|\bx-\boy|^{-m}\\
=&(8mC_2+C_1)|\bx-\boy|^{-m},
\end{align*}
and 
\begin{align*}
|I|^2\leq&\bigg(\int_{\Omega}\int_{\Sm}\big\vert \partial_{\boy}E_k(\bx-\boy,\bu,\bov)f(\bx,\bov)\big\vert dS(\bov)d\bx\bigg)^2\\
\leq &\bigg(\int_{\Omega}\int_{\Sm} 2^{\frac{m}{2}}\big\vert \partial_{\boy}E_k(\bx-\boy,\bu,\bov)\big\vert \cdot \big\vert f(\bx,\bov)\big\vert dS(\bov)d\bx\bigg)^2\\
\leq& \bigg((8mC_2+C_1)\int_{\Omega}\int_{\Sm}|\bx-\boy|^{-m}\big\vert f(\bx,\bov)\big\vert dS(\bov)d\bx\bigg)^2.
\end{align*}
By the Calderon and Zygmund Theorem in \cite[Theorem 3.1, XI]{MP}, we obtain
\begin{align}\label{Inorm}
&\int_{\Omega}\int_{\Sm}|I|^2dS(\bu)d\boy\nonumber\\
\leq& \int_{\Omega}\int_{\Sm} \bigg((8mC_2+C_1)\int_{\Omega}\int_{\Sm}|\bx-\boy|^{-m}\big\vert f(\bx,\bov)\big\vert dS(\bov)d\bx\bigg)^2dS(\bu)d\boy\nonumber\\
\leq &(8mC_2+C_1)^2\int_{\Sm}\int_{\Sm}\int_{\Omega}\bigg(\int_{\Omega}|\bx-\boy|^{-m}|f(\bx,\bov)|d\bx\bigg)^2d\boy dS(\bov)dS(\bu)\nonumber\\
\leq& (8mC_2+C_1)^2\int_{\Sm}\omega_{m-1}\int_{\Omega}|f(\bx,\bov)|^2d\bx dS(\bov)dS(\bu)\nonumber\\
=&(8mC_2+C_1)^2\omega_{m-1}^2\int_{\Omega}\int_{\Sm}|f(\bx,\bov)|^2d\bx dS(\bov)\nonumber\\
=&(8mC_2+C_1)^2\omega_{m-1}^2\|f\|_{L^2(\Omega\times\Bm,\Mk^+(\bu))}^2.
\end{align}
Also, we have
\begin{align}\label{IInorm}
&\int_{\Omega}\int_{\Sm}|II|^2dS(\bu)d\boy\nonumber\\
=&\int_{\Omega}\int_{\Sm}|f(\boy,\bu)|^2dS(\bu)d\boy=\|f\|_{L^2(\Omega\times\Bm,\Mk^+(\bu))}^2.
\end{align}
To estimate $|III|^2$, we notice that $\bt,\bu\in\Sm$ and $\bbeta=\bt\bu\bt\in\Sm$, which give us that
\begin{align*}
&|III|\\
=&\bigg\vert c_{m,k}^{-1}\int_{\Sm}\frac{2\overline{\bt}}{2-m}\bigg[\overline{\bu}\langle\bt,\overline{\partial}_{\bbeta}\rangle+\langle\bt,\overline{\bu}\rangle\partial_{\bbeta}+2\overline{\bt}\langle\bt,\overline{\bu}\rangle\langle\bt,\overline{\partial}_{\bbeta}\rangle\bigg]f(\boy,\bt\bu\bt)d\sigma(\bt)\bigg\vert\\
\leq&c_{m,k}^{-1}\int_{\Sm}\frac{2}{2-m}\cdot4\sum_{i=0}^{m-1}\bigg\vert\frac{\partial f(\boy,\bbeta)}{\partial \eta_i}\bigg\vert d\sigma(\bt).
\end{align*}
Since $f(\boy,\bbeta)$ is monogenic in $\bbeta$ with homogeneity $k$, this also implies that $f$ is also harmonic in $\bbeta$ with homogeneity $k$. Then, let $\bzeta=\frac{\bbeta}{2}\in\frac{1}{2}\Sm$ and $r=\frac{1}{4}$ in \eqref{harmonicestimate}, we have
\begin{align*}
\bigg\vert\frac{\partial f(\boy,\bbeta)}{\partial \eta_i}\bigg\vert=2^{k-1}\bigg\vert\frac{\partial f(\boy,\bzeta)}{\partial \zeta_i}\bigg\vert\leq \frac{m2^{2m+k+2}}{\sqrt{V(m)}}\|f(\boy,\cdot)\|_{L^2(\Bm)}.
\end{align*}
Therefore, we have
\begin{align}\label{IIInorm}
&\int_{\Omega}\int_{\Sm}|III|^2dS(\bu)d\boy\leq  \int_{\Omega}\int_{\Sm}\frac{m^22^{4m+2k+4}}{V(m)}\|f(\boy,\cdot)\|^2_{L^2(\Bm)}dS(\bu)d\boy\nonumber\\
=&\frac{m^22^{4m+2k+4}}{V(m)}\int_{\Omega}\int_{\Sm}\int_{\Bm}|f(\boy,\bov)|^2d\bov dS(\bu)d\boy\nonumber\\
=&\frac{m^22^{4m+2k+4}\omega_{m-1}}{V(m)}\int_{\Omega}\int_0^1\int_{\Sm}|f(\boy,r\bgamma)|^2r^{m-1}drdS(\bgamma)d\boy\nonumber\\
=&\frac{m^22^{4m+2k+4}\omega_{m-1}}{V(m)}\int_{\Omega}\int_{\Sm}\int_0^1\int_{\Sm}|f(\boy,\bgamma)|^2r^{m+2k-1}drdS(\bgamma)d\boy\nonumber\\
=&\frac{m^32^{4m+2k+4}}{m+2k}\int_{\Omega}\int_{\Sm}|f(\boy,\bgamma)|^2dS(\bgamma)d\boy\nonumber\\
=&\frac{m^32^{4m+2k+4}}{m+2k}\|f\|^2_{L^2(\Omega\times\Bm,\Mk^+(\bu))},
\end{align}
where the last second equality comes from the fact that $V(m)=\frac{\omega_{m-1}}{m}$. Plugging \eqref{Inorm},\eqref{IInorm},\eqref{IIInorm} into \eqref{Pinorm}, we obtain
\begin{align*}
&\|\Pi f\|_{L^2(\Omega\times\Bm,\Mk^-(\bu))}^2\leq \int_{\Omega}\int_{\Sm}2(|I|^2+|II|^2+|III|^2)dS(\bu)d\boy\\
\leq&2\bigg[(8mC_2+C_1)^2\omega_{m-1}^2\|f\|_{L^2(\Omega\times\Bm,\Mk^+(\bu))}^2+\|f\|_{L^2(\Omega\times\Bm,\Mk^+(\bu))}^2\\
&+\frac{m^32^{4m+2k+4}}{m+2k}\|f\|^2_{L^2(\Omega\times\Bm,\Mk^+(\bu))}\bigg]\\
=&\bigg(2(8mC_2+C_1)^2\omega_{m-1}^2+\frac{m^32^{4m+2k+5}}{m+2k}+2\bigg)\|f\|_{L^2(\Omega\times\Bm,\Mk^+(\bu))}^2,
\end{align*}
which gives us the norm estimate as the following
\begin{align*}
&\|\Pi f\|_{L^2(\Omega\times\Bm,\Mk^-(\bu))}\\
\leq& \sqrt{2(8mC_2+C_1)^2\omega_{m-1}^2+\frac{m^32^{4m+2k+5}}{m+2k}+2}\cdot \|f\|_{L^2(\Omega\times\Bm,\Mk^+(\bu))},
\end{align*}
and this completes the proof.
\end{proof}
\begin{remark}
A similar argument can give the same norm estimate for the $\Pi^{\dagger}$-operator.
\end{remark}
Now, we introduce some properties for the higher spin $\Pi$-operator as follows.
\begin{proposition}\label{propA}
Let $f\in W_2^1(\Omega\times\Bm,\Mk^+(\bu))$ and $g\in W_2^1(\Omega\times\Bm,\Mk^-(\bu))$. Then, we have
$$\Pi R_kf=R_k^{\dagger}f-R_k^{\dagger}F_kf,\ 
\Pi^{\dagger}R_k^{\dagger}g=R_kg-R_kF_k^{\dagger}g.$$
\end{proposition}
\begin{proof}
Using the Borel-Pompeiu formula with $R_k$ in Theorem \ref{BPF}, we obtain
\begin{align*}
\Pi R_kf=R_k^{\dagger}T_kR_k f=R_k^{\dagger}(I-F_k)f=R_k^{\dagger}f-R_k^{\dagger}F_kf.
\end{align*}
Similarly, the Borel-Pompeiu formula with $R_k^{\dagger}$ in Theorem \ref{BPF} gives us that
\begin{align*}
\Pi^{\dagger} R_k^{\dagger}g=R_kT_k^{\dagger}R_k^{\dagger} g=R_k(I-F_k^{\dagger})g=R_kg-R_kF_k^{\dagger}g,
\end{align*}
which completes the proof.
\end{proof}
\begin{remark}
	One might notice that some properties of the Dirac operator case given in \cite[Theorem 5]{GK} are not true for the $R_k$ case, since $R_kR_k^{\dagger}\neq R_k^{\dagger}R_k$ in general.
\end{remark}
The Proposition above also leads to the following mapping property.
\begin{proposition}
Let $\Omega\subset\R^m$ be a bounded domain. Then, the relations
\begin{align*}
&\Pi:R_k\left(\mathring{W}_2^1(\Omega\times\Bm,\Mk^+(\bu))\right)\longrightarrow R_k^{\dagger}\left(\mathring{W}_2^1(\Omega\times\Bm,\Mk^+(\bu))\right),\\
&\Pi^{\dagger}: R_k^{\dagger}\left(\mathring{W}_2^1(\Omega\times\Bm,\Mk^+(\bu))\right)\longrightarrow R_k\left(\mathring{W}_2^1(\Omega\times\Bm,\Mk^+(\bu))\right)
\end{align*} 
are true.
\end{proposition}
\begin{proof}
Let $f\in \mathring{W}_2^1(\Omega\times\Bm,\Mk^+(\bu))$ and $g\in \mathring{W}_2^1(\Omega\times\Bm,\Mk^-(\bu))$, then we know that $F_k f=F_k^{\dagger}g=0$. Proposition \ref{propA} immediately gives us the desired results.
\end{proof}
A connection between $\Pi$ and $\Pi^{\dagger}$ is given as follows.
\begin{theorem}\label{ThmId}
Let $\Omega\subset\R^m$ be a bounded domain and $0\leq s<\infty$. Suppose $f\in W^s_2(\Omega\times\Bm,\Mk^+(\bu))$ and $g\in W^s_2(\Omega\times\Bm,\Mk^-(\bu))$, Then, we have
\begin{align*}
&\Pi^{\dagger}\Pi f=f-R_kF_k^{\dagger}T_kf,\\
&\Pi\Pi^{\dagger} g=g-R_k^{\dagger}F_kT_k^{\dagger}g.
\end{align*}
\end{theorem}
\begin{proof}
These statements can be justified by a straightforward calculation. Indeed, we have
\begin{align*}
&\Pi^{\dagger}\Pi f=R_kT_k^{\dagger}R_k^{\dagger}T_kf=R_k(I-F_k^{\dagger})T_kf\\
=&(R_kT_k-R_kF_k^{\dagger}T_k)f=(I-R_kF_k^{\dagger}T_k)f=f-R_kF_k^{\dagger}T_kf,\\
&\Pi\Pi^{\dagger} g=R_k^{\dagger}T_kR_kT_k^{\dagger}g=R_k^{\dagger}(I-F_k)T_k^{\dagger}g\\
=&(R_k^{\dagger}T_k^{\dagger}-R_k^{\dagger}F_kT_k^{\dagger})g=(I-R_k^{\dagger}F_kT_k^{\dagger})g=g-R_k^{\dagger}F_kT_k^{\dagger}g,
\end{align*}
which completes the proof.
\end{proof}
Next, we investigate the adjoint operator of $\Pi$, denoted by $\Pi^*$, with respect to the following $\Clm$-valued inner product
\begin{align*}
\langle f, g\rangle=\int_{\Omega}\int_{\Sm}\overline{f(\boy,\bu)}g(\boy,\bu)dS(\bu)d\boy,
\end{align*}
where $f,g\in L^2(\Omega\times\Bm,\Mk^+(\bu))$.
\begin{theorem}
Let $f\in \mathring{W}_2^s(\Omega\times\Bm,\Mk^+(\bu))$ and $g\in \mathring{W}_2^s(\Omega\times\Bm,\Mk^-(\bu))$ with $s\geq 1$. Then, we have
\begin{align*}
\Pi^*f=T_k^{\dagger}R_kf,\ \Pi^{\dagger*}g=T_kR_k^{\dagger}g.
\end{align*}
\end{theorem}
\begin{proof}
From the definition of the higher spin $\Pi$-operator and the $L^2$ adjoint operator, we have
\begin{align*}
\langle \Pi f, g\rangle=\langle R_k^{\dagger}T_kf,g\rangle=\langle T_kf, R_k^{\dagger*}g\rangle=\langle f,T_k^*R_k^{\dagger*}g\rangle.
\end{align*}
We know that
\begin{align*}
&\langle T_k f,g\rangle=\int_{\Omega}\int_{\Sm}\overline{T_k f(\boy,\bu)}g(\boy,\bu)dS(\bu)d\boy\\
=&-\int_{\Omega}\int_{\Sm}\bigg(\int_{\Omega}\int_{\Sm}\overline{E_k(\bx-\boy,\bu,\bov)f(\bx,\bov)}dS(\bov)d\bx\bigg)g(\boy,\bu)dS(\bu)d\boy\\
=&-\int_{\Omega}\int_{\Sm}\bigg(\int_{\Omega}\int_{\Sm}\overline{f(\bx,\bov)}\cdot\overline{E_k(\bx-\boy,\bu,\bov)}dS(\bov)d\bx\bigg)g(\boy,\bu)dS(\bu)d\boy.
\end{align*}
According to \eqref{kernelconjugate}, we know
\begin{align*}
&\overline{E_k(\bx-\boy,\bu,\bov)}=c_{m,k}^{-1}\overline{Z_k\left(\frac{(\bx-\boy)\bu(\bx-\boy)}{|\bx-\boy|^2},\bov\right)}\frac{\bx-\boy}{|\bx-\boy|^m}\\
=&c_{m,k}^{-1}Z_k\left(\bov,\frac{(\bx-\boy)\bu(\bx-\boy)}{|\bx-\boy|^2}\right)\frac{\bx-\boy}{|\bx-\boy|^m}\\
=&c_{m,k}^{-1}\frac{\bx-\boy}{|\bx-\boy|^m}Z_k\left(\frac{(\overline{\bx-\boy})\bov(\overline{\bx-\boy})}{|\bx-\boy|^2},\bu\right)=E_k^{\dagger}(\bx-\boy,\bov,\bu)\\
=&-E_k^{\dagger}(\boy-\bx,\bov,\bu).
\end{align*}
Therefore, $\langle T_k f,g\rangle$ becomes
\begin{align*}
=&-\int_{\Omega}\int_{\Sm}\bigg(\int_{\Omega}\int_{\Sm}\overline{f(\bx,\bov)}\cdot E_k^{\dagger}(\bx-\boy,\bov,\bu)dS(\bov)d\bx\bigg)g(\boy,\bu)dS(\bu)d\boy\\
=&\int_{\Omega}\int_{\Sm}\overline{f(\bx,\bov)}\bigg(\int_{\Omega}\int_{\Sm}E_k^{\dagger}(\boy-\bx,\bov,\bu)g(\boy,\bu)dS(\bu)d\boy\bigg)dS(\bov)d\bx\\
=&-\int_{\Omega}\int_{\Sm}\overline{f(\bx,\bov)}(T_k^{\dagger}g(\bx,\bov))dS(\bov)d\bx=\langle f,-T_k^{\dagger}g\rangle,
\end{align*}
and this gives us that $T_k^*=-T_k^{\dagger}$. Furthermore, 
\begin{align*}
&\langle R_k f,g\rangle=\int_{\Omega}\int_{\Sm}\overline{R_k f(\boy,\bu)}f(\boy,\bu)dS(\bu)d\boy\\
=&\int_{\Omega}\int_{\Sm}\bigg(\overline{f(\boy,\bu)}R_k^{\dagger}\bigg)f(\boy,\bu)dS(\bu)d\boy\\
=&\int_{\Omega}\int_{\Sm}\overline{f(\boy,\bu)}\bigg(R_k^{\dagger}f(\boy,\bu)\bigg)dS(\bu)d\boy\\
=&\langle f,R_k^{\dagger}g\rangle,
\end{align*}
which implies that $R_k^*=R_k^{\dagger}$. Therefore, we have
\begin{align*}
\langle \Pi f, g\rangle=\langle f,T_k^*R_k^{\dagger*}g\rangle=\langle f,-T_k^{\dagger}R_kg\rangle,
\end{align*}
which says that $\Pi^*=T_k^{\dagger}R_k$. Similar argument can gives us that $\Pi^{\dagger *}=T_kR_k^{\dagger}$, which completes the proof.
\end{proof}

\section{Existence of solutions to the higher spin Beltrami equation}
In \cite{CK,GK}, the authors used the norm estimate of a generalized $\Pi$-operator to establish the existence of solutions to a hypercomplex Beltrami equation. In this section, we will firstly introduce a higher spin Beltrami equation in the framework of Clifford analysis, then a sufficient condition for existence of solutions to the higher spin Beltrami equation will be provided.
\par
Let $\Omega\subset\R^m$ be a bounded domain with smooth boundary $\partial\Omega$ and $f\in C^1(\Omega\times\Bm,\Mk^+(\bu))$.
\ Moreover, let $\omega\in C^1(\Omega\times\Bm,\Mk^+(\bu))$. Then, we call the equation
\begin{align}\label{I1}
R_k\omega=fR_k^{\dagger}\omega
\end{align}
a \emph{higher spin Beltrami equation}. In particular, when $k=0$, all terms on the variable $\bu$ disappear, it reduces to the classical Beltrami equation. Before, we introduce existence of solutions to the higher spin Beltrami equation, we need to following technical lemma.
\begin{lemma}
Suppose $\omega\in W_2^1(\Omega\times\Bm,\Mk^+(\bu))$. Therefore exists a function $\phi\in L^2(\Omega\times\Bm,\Mk^+(\bu))$ satisfying $R_k\phi=0$ and $h\in L^2(\Omega\times\Bm,\Mk^+(\bu))$ such that
\begin{align}\label{I2}
\omega=\phi+T_kh.
\end{align}
\end{lemma}
\begin{proof}
Given $\omega\in W_2^1(\Omega\times\Bm,\Mk^+(\bu))$, let $h=R_k\omega$, the mapping properties in Proposition \ref{Rkbound} tells us that $h\in L^2(\Omega\times\Bm,\Mk^+(\bu))$. Furthermore, let $\phi=\omega-T_k h$, we immediately have $R_k\phi=R_k\omega-R_kT_k h=0$, which completes the proof. 
\end{proof}
Now, we plug the substitution \eqref{I2} in to \eqref{I1} to obtain
\begin{align*}
R_k(\phi+T_k h)=f\big(R_k^{\dagger}(\phi+T_k h)\big),
\end{align*}
since $T_k$ is the right inverse of $R_k$, and $R_k\phi=0$, the equation above becomes
\begin{align}\label{I3}
h=f\big(R_k^{\dagger}\phi+\Pi h\big).
\end{align}
This implies that, on the one hand, $\omega$ is a solution of \eqref{I1}, if $h$ is a solution of \eqref{I3}. On the other hand, each solution of \eqref{I1} can be represented by \eqref{I2}.
\par
Before introducing our main result in this section, we review the Banach fixed-point theorem as follows.
\begin{theorem}[Banach Fixed-Point Theorem]\cite{BFPT}
	Let $(X,d)$ be a non-empty complete metric space. A mapping $T: X\longrightarrow X$ is called a contraction mapping on $X$ if there exists $q\in[0,1)$, such that $d(T(x),T(y))\leq qd(x,y)$. Such $T$ admits a unique fixed-point $x^*\in X$, which means $T(x^*)=x^*$.
\end{theorem}
Now, we claim that
\begin{theorem}
Let $f\in L^{\infty}(\Omega\times\Bm,\Mk^+(\bu))$, and $\|f\|_{L^{\infty}(\Omega\times\Bm,\Mk^+(\bu))}\leq \rho< C^{-1}$, where $C$ is given in Theorem \ref{BdPi}. Then, our singular integral equation \eqref{I3} has a unique solution $h\in \mathcal{L}^2(\Omega\times\Bm,\Mk^+(\bu))$. Furthermore, our higher spin Beltrami equation \eqref{I1} also has a solution given by $\omega=\phi+T_kh$.
\end{theorem}
\begin{proof}
We denote $$T:\ h\rightarrow fR_k^{\dagger}\phi+f\Pi h,$$ and in our case, we have
	\begin{align*}	
		&||T(h_1)-T(h_2)||_{L^2(\Omega\times\Bm,\Mk^+(\bu))}\\
		=&||f(R_k^{\dagger}\phi+\Pi h_1)-f(R_k^{\dagger}\phi+\Pi h_2)||_{L^2(\Omega\times\Bm,\Mk^+(\bu))}\\
		&\leq||f\Pi ||_{L^2(\Omega\times\Bm,\Mk^+(\bu))}\cdot||h_1-h_2||_{L^2(\Omega\times\Bm,\Mk^+(\bu))}.
	\end{align*}
	Therefore, with the Banach Fixed-Point Theorem, when 
	\begin{align*}		
	|| f||_{L^{\infty}(\Omega\times\Bm,\Mk^+(\bu))}\leq\rho<C^{-1},
\end{align*}
$T$ is contractive, which leads to the existence of a unique solution to \eqref{I3}. Applying our norm estimate in Theorem \ref{BdPi} for the $\Pi$-operator we get the condition 
 	\begin{align*}	
 	|| f||_{L^{\infty}(\Omega\times\Bm,\Mk^+(\bu))}\leq\rho<C^{-1}
 \end{align*}being sufficient for the existence of a solution of equation \eqref{I3}. The proof is completed.
\end{proof}
\subsection*{Acknowledgments}
The work of Chao Ding is supported by the National Natural Science Foundation of China (No. 12271001).
\subsection*{Disclosure statement}
No potential conflict of interest was reported by the author. Data sharing not applicable to this article as no datasets were generated or analyzed during the current study.

\end{document}